\documentclass[a4paper,12pt]{amsart}

\usepackage{graphicx}
\usepackage[english]{babel}

\usepackage[T1]{fontenc}
\usepackage[utf8]{inputenc}
\usepackage{csquotes}
\usepackage{verbatim}
\usepackage{todonotes}
\usepackage{amsmath}
\usepackage{amsthm}
\usepackage{amsfonts, amssymb, fancybox, multicol, array, tabularx, amscd}
\usepackage{enumerate}
\usepackage{psfrag}
\newtheorem{theorem}{Theorem}[section]

\newtheorem{proposition}[theorem]{Proposition}
\newtheorem{corollary}[theorem]{Corollary}

\newtheorem{remark}[theorem]{Remark}

\theoremstyle{definition}

\theoremstyle{remark}

\newcommand{\R}{\mathbb{R}}
\newcommand{\N}{\mathbb{N}}
\newcommand{\Z}{\mathbb{Z}}
\newcommand{\runo}{\mathbb{R}}
\newcommand{\rdue}{\mathbb{R}^2}
\newcommand{\rtre}{\mathbb{R}^3}

\newcommand{\e}{\varepsilon}
\newcommand{\f}{\varphi}
\newcommand{\matrici}{\mathbb{M}^{3 \times 3}}
\newcommand{\simm}{\mathbb{M}^{3 \times 3}_{sym}}

\newcommand{\rotazioni}{SO(3)}

\def\dsp{\displaystyle}

\newcommand{\va}[1]{\left| #1 \right|}  
\newcommand{\su}[2]{{\left. #1 \right|}_{#2}}

\DeclareMathOperator{\dist}{dist}

\DeclareMathOperator{\curl}{curl}
\newcommand{\lploc}{\mathnormal{L}^p_{\rm loc}}
\newcommand{\lp}{\mathnormal{L}^p}

\newcommand{\AS}{\mathcal{AS}(\Gamma, \mathbf B)}
\newcommand{\AD}{\mathcal{AD}}
\newcommand{\ADM}{\mathcal{ADM}}
\newcommand{\tot}{E_{\alpha,R}^{tot}}
\newcommand{\el}{E_{\alpha,R}^{el}}
\newcommand{\Om}{\Omega}
\newcommand{\dpr}{{\mathcal D}'}

 \title[A variational model for dislocations at semi-coherent interfaces]{A variational model for dislocations at semi-coherent interfaces}
 
 \author[S. Fanzon]
 {Silvio Fanzon}
 \address[Silvio Fanzon]{University of Sussex, Department of Mathematics, Pevensey 2 Building, Falmer Campus,
Brighton BN1 9QH, United Kingdom}
 \email{S.Fanzon@sussex.ac.uk}

\author[M. Palombaro]
 {Mariapia Palombaro}
 \address[Mariapia Palombaro]{University of Sussex, Department of Mathematics, Pevensey 2 Building, Falmer Campus,
Brighton BN1 9QH, United Kingdom}
 \email{M.Palombaro@sussex.ac.uk}

\author[M. Ponsiglione]
 {Marcello Ponsiglione}
 \address[Marcello Ponsiglione]{Dipartimento di Matematica, Sapienza Universit\`a di Roma, 00185 Roma, Italy}
 \email{ponsigli@mat.uniroma1.it}

\begin{document}

\begin{abstract}
\small{
We propose and analyze a simple variational model for dislocations at semi-coherent interfaces. The energy functional describes  the competition between two terms: a surface energy induced by dislocations that compensate the lattice misfit at the interface,  and a far field elastic energy, spent to  decrease the amount of  needed dislocations. 
We prove that the former scales like the surface area of the interface, the latter like its diameter.  

The proposed continuum model is
deduced from some rigorous derivation from the semi-discrete theory of dislocations. Even if we deal with finite elasticity, linearized elasticity naturally emerges in our analysis since the far field strain vanishes as the interface size increases.


}
\end{abstract}

\maketitle
\tableofcontents

\section*{Introduction}\label{introduction}
Dislocations are line topological defects in the periodic structure of crystals. Their motion represents the microscopic mechanism of plastic flow, while their presence at grain boundaries decreases the energy induced by lattice misfits. 

In this paper we propose and analyze a variational model describing dislocations at semi-coherent interfaces, focusing on flat two dimensional interfaces between two crystalline materials with different underlying lattice structures $\Lambda^+$ and $\Lambda^-$.  Specifically, we assume that the lattice  $\Lambda^+$ 
(lying on top of $\Lambda^-$) is a dilation with factor $\alpha>1$ of $\Lambda^-$. We are interested in semi-coherent interfaces, corresponding to small misfits $\alpha \approx 1$.

Since in the reference configuration (where both crystals are in equilibrium) the density of the atoms of $\Lambda^+$ is lower than that of $\Lambda^-$,   in the vicinity of the interface  there are many atoms  having the ``wrong'' number of first neighbors (see first picture in Figure \ref{freedisloc}). Such atoms form line singularities (relatively closed paths lying on the interface), which correspond to edge dislocations. The crystal can reduce the number of such dislocations through a compression strain acting on $\Lambda^+$ near the interface, at the price of storing some far field elastic energy. 
A deformation that coincides with $x \mapsto \alpha^{-1}x$  near the interface would provide a defect-free perfect match between the crystal lattices (see third picture in Figure \ref{freedisloc}). 
In fact, experimental evidence suggets that the true deformed configuration is  the result of a balance (see middle picture in Figure~\ref{freedisloc}) between the elastic energy spent to match the crystal structures and the dislocation energy spent to release the far field  elastic energy, with the former scaling (for defect free configurations) like the volume of the body and the latter  
like the surface  area of the interface.

\begin{figure}[h]
\begin{center}
{\includegraphics[scale=0.75]{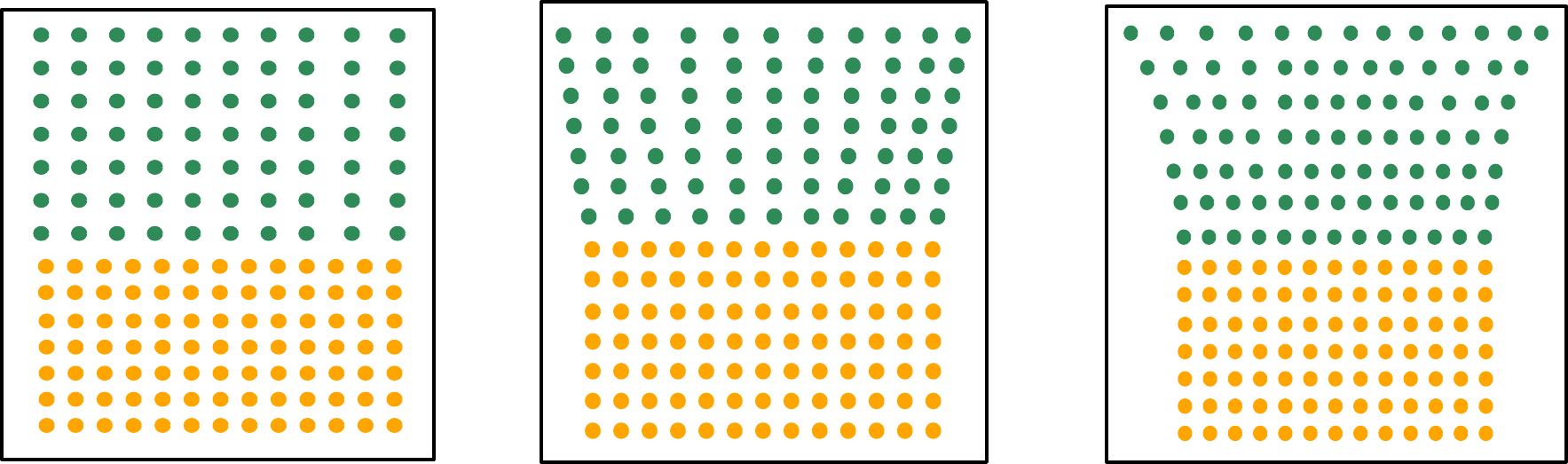}}
\caption{\label{freedisloc} Left: A bulk stress-free configuration. Right: a defect-free configuration. Center: a schematic picture of a true energy minimizer.}   
\end{center}
\end{figure}

This is why the common perspective of the scientific community working on this problem has been to understand which configurations of dislocations minimize the elastic stored energy, and much effort has been devoted to describe those configurations  for which the dislocation energy contribution is predominant, and the far field elastic energy is negligible (\cite{ShRe}, \cite{Hirsch}).  
As a matter of fact,  for  large crystals,  periodic patterns of edge dislocations are observed at  interfaces.

Here, we propose  a simple variational model to analyze the competition between surface and elastic energy.  We show that, for large interfaces, the dislocation energy of minimizers scales like the area of the interface, while the elastic far field energy like its diameter.   

The proposed model is not purely discrete; indeed it is a continuum model that stems from considerations based on a rigorous derivation from the so called semi-discrete theory of dislocations.  
In this framework, the reference configuration is described as a continuum, while dislocations are represented by one dimensional singularities of the strain. In single crystals, the energy induced by straight  edge dislocations  has a logarithmic tail, which diverges as the ratio between the crystal size and the atomic distance tends to $+\infty$. The $\Gamma$-convergence analysis for these systems as the atomic distance tends to zero has been recently done in \cite{DGP}, showing that dipoles as well as isolated dislocations do not contribute to decrease the elastic energy, so that in single crystals only the so called {\it geometrically necessary dislocations} are good competitors in the energy minimization. 

Quite different is the case of polycrystals treated in this paper, where dislocations contribute to decrease the elastic energy. 
A rigorous variational  justification of dislocation nucleation in heterostructured nanowires was obtained by M\"uller and Palombaro \cite{mp} in the context of nonlinear elasticity.  The model 
proposed in \cite{mp} was later generalized to a discrete to continuum setting in \cite{LPS1,LPS2} 
(see also \cite{ALP} for recent advancements in the microscopic setting). 
A variational model for misfit dislocations in elastic thin films, in connection with epitaxial growth,  has been recently proposed in \cite{FFLM}. 

In the first part of the paper we set and analyze the problem in the semi-discrete framework, which provides the theoretical background for the proposed continuum model.   
In the semi-discrete model, the reference configuration of
the hyperelastic body is the cylindrical region $\Omega_r:=S_r \times (-hr,hr)$, 
where $r,h>0$ and $S_r:=[-r/2,r/2]^2$.
The interface $S_r \times \{0\}$ separates the two regions of the body, $\Omega_r^-:=S_r \times (-hr,0)$ and $\Omega_r^+:=S_r\times (0,hr)$, 
with underlying  crystal structures $\Lambda^-$ and $\Lambda^+$ respectively. 
We will refer to  $\Omega_r^-$ and  $\Omega_r^+$ as the underlayer and overlayer, respectively.
We  assume that the material equilibrium is the identity $I$ in $\Omega^-_r$ (implying that the underlayer is already in equilibrium) and $\alpha I$ in
 $\Omega^+_r$, where $\alpha > 1$ measures the misfit between the two lattice parameters. Notice that the identical deformation of $\Omega_r$, 
which corresponds to a dislocation-free configuration,
 is not stress-free, since the overlayer is not in equilibrium. 
  Furthermore, in order to simplify the analysis, we assume that $\Omega^-_r$ is rigid, so that 
 only $\Omega_r^+$ is subjected to deformations. 
 
 We assume that deformations try to minimize a stored elastic energy (in $\Omega^+_r$), whose density is described by a nonlinear frame indifferent function $W \colon \matrici \to [0,+\infty)$.
    In classical finite elasticity, $W$ acts on deformation gradients $F:=\nabla v$. In this model dislocations are introduced as line defects of the strain: more precisely,    
we allow the strain field $F$ to have a non vanishing curl, concentrated on  dislocation lines on the interface $S_r$. 
Therefore, the admissible strains are maps
 $F \in \lp (\Omega_r;  \matrici)$ (where $p<2$ is fixed, according to the growth assumptions on $W$) that satisfy
 \begin{equation} \label{salto0}
 \curl F = \sum_i  - \mathbf{b}_i \otimes \dot{\gamma_i} \, d \mathcal{H}^1 \llcorner \gamma_i 
 \end{equation}
in the sense of measures and such that $F=I$ in $\Omega^-_r$. Here $\{\gamma_i\}$ is a finite collection of closed curves,  
and $\mathbf{b}_i \in \rtre$
denotes the Burgers vector, which is constant on each $\gamma_i$.
The Burgers vector belongs to the set of slip directions, which is a given material property of the crystal. 
We assume that 
the slip directions are given by $b \Z \{e_1,e_2\}$, where $b>0$ represents the lattice spacing of $\Lambda^-$, 
and that the dislocation curves $\gamma_i$ have support on the  grid $\left[\big((b\Z \times \R) \cup (\R\times b\Z)\big) \cap S_r \right]\times \{0\}$. 
Notice that this choice is consistent with the cubic crystal structure, and that $b$ is independent of $r$, i.e., independent of the size of the body.  

In Section \ref{mod1} we study the asymptotic behaviour of minimizers of the elastic energy functional with respect to all possible pairs of compatible (i.e., satisfying \eqref{salto0}) strains and dislocations, refining the analysis first done in \cite{mp}. 
In Proposition \ref{prop:dislo} we show that, as $r \to +\infty$, the elastic energy of minimizers per unit area of the interface tends to a given 
energy surface density $E_\alpha$.  
As a consequence, we show that there exists a critical $r^*$ such that, for larger size of the interface, dislocations are energetically 
favorable (see Theorem \ref{thm:dislo favorevoli}). The proof of these results is based on an explicit construction of an array of dislocations (see Figure \ref{lungh0}) and of admissible fields, which is optimal in the energy scaling (see Proposition \ref{prop:doppio cilindro}). While we could guess  that the dislocation configuration is somehow optimal, the strains that we define are surely not, so that our construction does not provide the sharp formula for the surface energy density $E_\alpha$, which depends on the specific form of the elastic energy density $W$. 
\begin{figure}[h]
\begin{center}
{\includegraphics[scale=0.55]{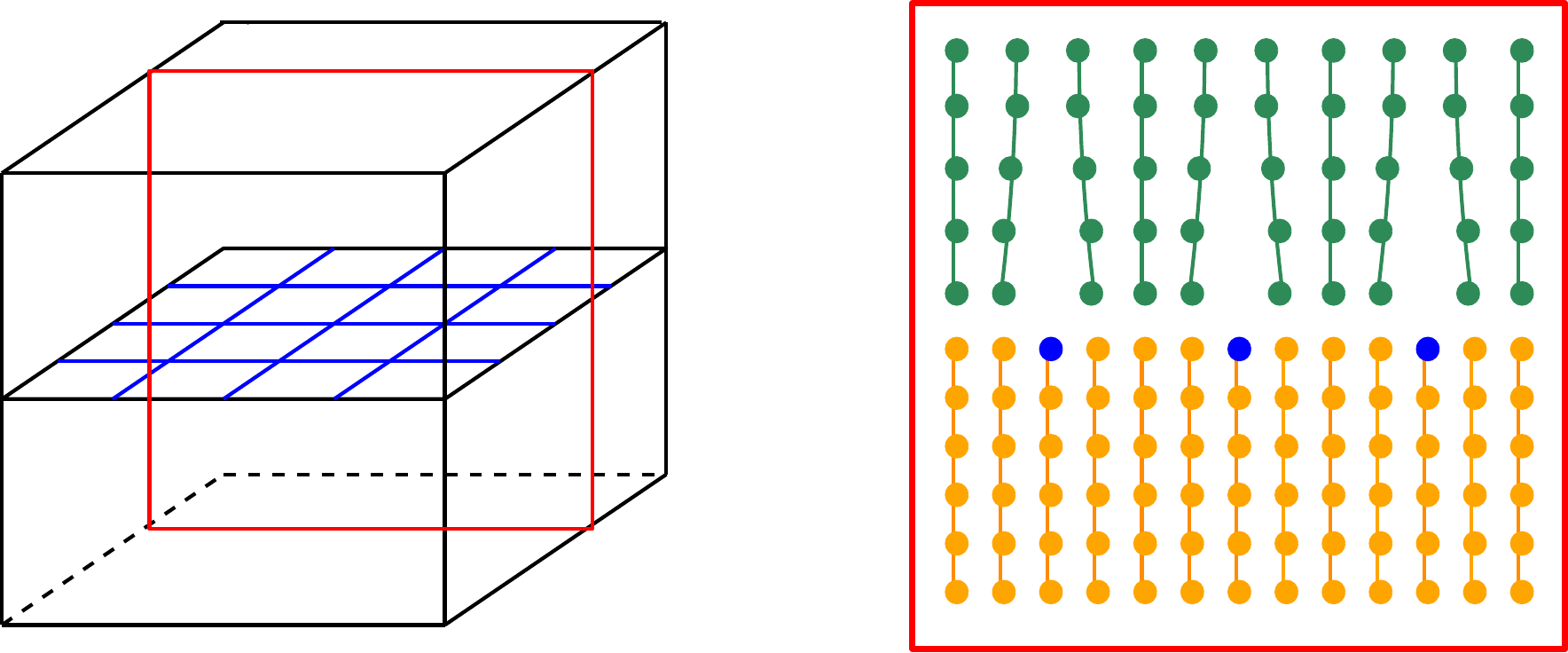}}
\caption{\label{lungh0} Left: The 3D crystal. Right: a 2D cross-section with the  lattice structure. Red squares: The 2D slice. Green: The lattice $\Lambda^+$. Orange: The lattice $\Lambda^-$. Blue: Edge dislocations.}   
\end{center}
\end{figure}

In the second part of the paper we propose a simple continuous model for dislocations at semicoherent 
interfaces, describing in particular heterogeneous nanowires. 
Although we deal with a continuum model, our approach is consistent with the discrete analysis developed in \cite{LPS1,LPS2}. In this model 
we work with actual gradient fields far from the interface, where the curl of the strain is now a diffuse measure, in contrast with  \eqref{salto0}. Dislocation nucleation is 
taken into account by introducing a free parameter into the total energy and eventually optimizing on it.
Specifically, we assume that the underlayer occupies the cylindrical region $\Om_R^-$ (which is fixed), 
while the reference configuaration of the overlayer 
is $\Om_r^+$,  where $r=\theta R$ and $\theta \in(0, 1)$ is a free parameter in the total energy functional. 
The class of admissible deformation maps is defined by 
\begin{equation}
\ADM_{\theta,R} :=\left\{v\in W^{1,2}(\Omega^+_{r}; \rtre) \, \colon \,  v (x) = \frac{1}{\theta} \, x \, \text{ on }  S_{ r} \right\}.
\end{equation}
In this way $v(S_{ r})=S_R$ for all $v\in \ADM_{\theta, R}$. 
In view of the analysis performed in the semi-discrete setting, the area of 
$S_R \smallsetminus S_{ r}$ divided by $b$  can be interpreted as the total dislocation length. 
This suggests to introduce the plastic energy defined by
\begin{equation*}
E^{pl}_{R}(\theta):= \sigma r^2 ({\theta}^{-2}-1) = \sigma R^2 (1-{\theta}^2).
\end{equation*}
Here $\sigma >0$ is a given  material constant of the crystal, which multiplied by $b$ represents the energetic cost of 
dislocations per unit length.  
In principle, $\sigma$ could be derived starting from the surface energy density, yielding in the limit of vanishing misfit $\dsp\sigma= \lim_{\alpha\to 1} \frac{E_\alpha}{\alpha^2-1}$ (see \eqref{conj}).
Alternatively,
$\sigma$ can be expressed in terms of the Lam\'e moduli of the linearized elastic tensor corresponding to $W$ and of the (unknown) chemical core energy density $\gamma^{ch}$ induced by dislocations 
(see \eqref{sigma?} in Section \ref{conjr}). The latter contribution is implicitly taken into account by the nonlinear energy density $W$ in finite elasticity. 

Based on the previous considerations, our goal is to study the total energy functional defined by
\[
\tot (\theta,v):= E^{el}_{\alpha,R}(\theta,v)  + E^{pl}_{R}(\theta) =  \int_{\Omega_{r}^+ }  W (\nabla v (x)) \, dx + \sigma R^2 (1-{\theta}^2) ,
\]
for $v\in \ADM_{\theta,R}$.
Set 
\begin{equation*}
\el (\theta):= \inf \left\{E^{el}_{\alpha,R}(\theta,v)  \, \colon \,   v \in \ADM_{\theta, R} \right\}  , \qquad \tot (\theta) := \el (\theta) + E^{pl}_{R}(\theta).
\end{equation*}
Notice that if $\theta=1$, then no dislocation energy is present, i.e. $\tot (1)= \el (1)$. Instead, if 
$\theta= \alpha^{-1}$ no elastic energy is stored (since $v(x):= \alpha x$ is admissible and $W(\alpha I)=0$). 

The remaining and main part of the paper is devoted to the analysis of almost minimizers of  $\tot$, as $R\to +\infty$. In Theorem \ref{il teorema} we show that the optimal $\theta_R$ tends to $\alpha^{-1}$ from below, showing that the dislocation energy spent to release bulk energy is predominant, but still $\theta_R\neq \alpha^{-1}$, so that also a far field bulk energy is present (see Figure \ref{freedisloc}).

In order to compute the optimal $\theta_R$,
we perform a Taylor expansion (through a $\Gamma$-convergence analysis) of the plastic and elastic part of the energy, proving in particular, 
that the first scales like $R^2$, while the second like $R$. Prefactors in such energy expansions are computed, depending only on $\alpha, \, \sigma$ and on the fourth-order tensor obtained by 
linearizing $W$ about the identity. 

In conclusion, the proposed functional provides a simple prototypical variational model to describe the competition between the dislocation energy concentrated around the interface between 
materials with 
different crystal structures, and the far field elastic energy. 
This model fits into the class of free boundary problems, since  the overlayer is a variable in the minimization problem, though only through a scalar parameter representing its size.  Our formulation is quite specific, dealing with two lattices where one is a small dilation of the other. Therefore, it is meant to model semi-coherent interfaces between two different lattices, for example in heterostructured  nanowires.
Nevertheless,  our approach seems flexible enough to be adapted to more general situations, to model epitaxial crystal growth (where the surface energy of the free external boundary in contact with air should be added to the energy functional), and to more general interfaces, such as grain boundaries, where the misfit in the crystal structures is due to mutual  
rotations between the grains instead of dilations of the lattice parameters.

\section{A line defect model}\label{mod1}
\subsection{Description of the model}
We introduce a semi-discrete model for dislocations, which are described as line defects of the strain. 
 
Let $\Omega_1=S_1 \times (-h,h)$
be the reference configuration of a cylindrical hyperelastic body. Here $h>0$ is a fixed height and 
$S_1= \left\{ (x_1,x_2,0) \in \rtre \, \colon \, \va{x_1}, \va{x_2} < 1/2 \right\}$  
is a square of side one centered at the origin, separating  parts of the body with underlying 
crystal structures $\Lambda^-$ and $\Lambda^+:= \alpha \Lambda^-$, with $\alpha > 1$. 
For any given $r>0$, we will consider scaled versions of the body  $\Omega_r:= r\Omega_1$ and $S_r:= r S_1$. 
%


%

Set $\Omega_r^-:=S_r \times (-hr,0)$ and $\Omega_r^+:=S_r \times (0,hr)$.
We assume that the material equilibrium is the identity $I$ in $\Omega^-_r$ (which means that the material is already in equilibrium in  $\Omega^-_r$) and $\alpha I$ in
 $\Omega^+_r$.
 We are interested in small misfits, which generate so called semi-coherent interfaces; therefore, we will deal with $\alpha \approx 1$. 
 More specifically, we assume that the lattice distances of $\Lambda^-$ and $\Lambda^+$ are commensurable, and in particular that  
 $\alpha :=1+1/n$ for some given $n \in \N$.
Moreover, in order to simplify the analysis, we assume that $\Omega^-_r$ is rigid, namely, that 
the admissible deformations coincide with the identical deformation in $\Omega^-_r$.
 
 According to the hypothesis of hyperelasticity, we assume that the crystal tries to minimize a stored elastic energy (in $\Omega^+_r$), whose density is described by a function $W \colon \matrici \to [0,+\infty)$. We require that $W$ is continuous and frame indifferent,
i.e.,
\begin{equation} \label{frame ind}
W(F)=W(R \,F) \qquad \text{for every} \quad F \in \matrici, \, R \in \rotazioni \,.
\end{equation}
Moreover, there exist $p \in [1,2)$ and constants $C_1,C_2>0$ such that $W$ satisfies the following growth conditions:
\begin{equation} \label{growth1}
C_1 \left(  \,   \dist^2 (F, \alpha \rotazioni ) \wedge (\va{F}^p +1)  \right) \leq W(F) \leq C_2 
\left(  \,   \dist^2 (F,\alpha \rotazioni ) \wedge (\va{F}^p +1) \right)
\end{equation}
for every $F \in \matrici$.  

In absence of dislocations, the deformed configuration of the body can be described by a
sufficiently smooth deformation $v: \Omega_r^+ \to \rtre$. The corresponding elastic energy 
is given by
\begin{equation}\label{elene}
E^{el}(v):= \int_{\Omega_r^+} W(\nabla v) \, dx.
\end{equation}
The field $\nabla v$ is referred to as deformation strain.

We now explain how to introduce dislocations in the present model. 
As in \cite{mp}, dislocations are described by deformation strains  whose curl is not free, but
 concentrated on lines lying on  the
 interface $S_r$ between $\Omega^-_r$ and $\Omega^+_r$.  
  
 Assume for the time being that the dislocation line $\gamma \subset S_r$ is a Lipschitz, relatively closed curve in $S_r$. 
The latter condition implies that $\Omega_r \smallsetminus \gamma$ is not simply connected.
 Therefore, the strain is a  map
 $F \in \lp (\Omega_r;  \matrici)$ that satisfies
 \begin{equation} \label{salto}
 \curl F =  - \mathbf{b_\gamma} \otimes \dot{\gamma} \, d \mathcal{H}^1 \llcorner \gamma 
 \end{equation}
in the sense of distributions and $F=I$ in $\Omega^-_r$. 
\begin{figure} 
\centering   
\def\svgwidth{6cm}   
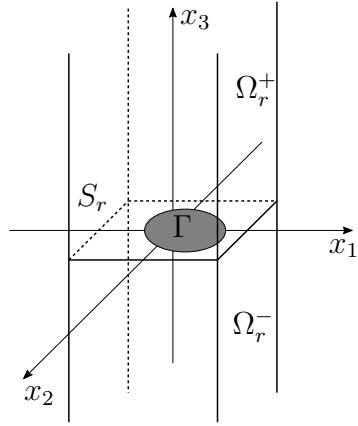  
\caption{Reference configuration $\Omega_r$\, .}   
\label{fig:cilindro}
\end{figure} 
The vector $\mathbf{b}_\gamma \in \rtre$
denotes the Burgers vector, which is constant on $\gamma$, and together with the dislocation line $\gamma$, uniquely
characterizes the dislocation. 
The Burgers vector belongs to the class of slip directions, which is a given material property of the crystal. As a further simplification, we assume that 
the slip directions are given by $b \Z \{e_1,e_2\}$, where $b>0$ represents the lattice spacing of the lower crystal $\Omega_r^-$. Notice that this choice is consistent with a cubic crystal structure, and that $b$ is independent of $r$, i.e., independent of the size of the body.  

If $\omega \subset \Omega_r \smallsetminus \gamma$ is a simply connected region, then \eqref{salto} implies that $\curl F = 0$ in $\dpr(\omega,\matrici)$ and therefore there exists $v \in \mathnormal{W}^{1,p}(\omega; \rtre)$ 
such that 
$F=\nabla v$ a.e. in $\omega$. 
Thus, any vector field $F$ satisfying \eqref{salto} is locally the gradient of a Sobolev map. 
In particular, if 
$\Sigma$ is a sufficiently smooth surface having $\gamma$ as its boundary, then one can find $v \in \mathnormal{SBV}_{\rm loc} (\Omega_r;\rtre)$ such that $F=\nabla v$, $v=x$ in $\Omega^-_r$ 
and its distributional gradient satisfies
\[
\mathnormal{D} v = \nabla v \, d x + \mathbf{b}_{\gamma} \otimes \nu \, d \mathcal{H}^2 \llcorner \Sigma
\] 
where $\nu$ is the unit normal to $\Sigma$.
That is, $F=\nabla v$ is the absolutely continuous part of the distributional gradient of $v$. 
As customery (see \cite{ortiz_vienna}), we interpret $F$ as the elastic part of the deformation $v$, so that the elastic energy induced by $v$ is given by
$$
E^{el}(v):= \int_{\Omega_r^+} W(F) \, dx.
$$
From now on we will assume that the dislocation curves have support in the grid $(b\Z \times \R) \cup (\R\times b\Z)$. Moreover, we will consider multiple dislocation curves. 
More precisely, we denote by $\AD$ the class of all admissible pairs $(\Gamma,\mathbf B)$, where $\Gamma$ is a finite collection of admissible closed curves $\{\gamma_i\}$, and $\mathbf B = \{\mathbf b_i\}$, $b_i \in b \Z \{e_1,e_2\}$, is the corresponding collection of Burgers vectors. 
Notice that each dislocation curve can be decomposed into ``minimal components'', i.e., we can always assume that $\gamma_i = \partial Q_i$, 
where $Q_i$ is a square of size $b$ with sides contained in the grid $(b\Z \times \R) \cup (\R\times b\Z)$. 
Given an admissible pair $(\Gamma,\mathbf B)$, we denote by 
$\mathbf{b} \otimes \dot{\gamma}(x)$ the field  that coincides with $\mathbf{b}_i \otimes \dot\gamma_i(x)$ if $x$ belongs to a single curve $\gamma_i$, and with 
$ \mathbf{b}_i \otimes \dot\gamma_i(x) + \mathbf{b}_j \otimes \dot\gamma_j(x)$ if 
$x$ belongs to two different curves $\gamma_i$ and $\gamma_j$.
For the sake of computational simplicity, whenever it is convenient  we will  assume 
\begin{equation}\label{simp}
\frac{r(\alpha -1)}{2b}\in\N.  
 \end{equation}
Recalling that $\dsp \alpha = 1 + \frac{1}{n}$, assumption \eqref{simp} implies that $\dsp \frac{r}{2b}\in\N$.
The set of admissible deformation strains $\AS$ associated with a given admissible dislocation $(\Gamma, \mathbf B)$ is then defined by
\begin{equation} \label{transizioni}
\AS := \left\{ F \in \lploc (\Omega_r; \matrici) \colon F=I \text{ in } \Omega^-_r, \,\curl F = - \mathbf{b} \otimes \dot{\gamma} \, d \mathcal{H}^1 \llcorner \Gamma  \right\},
\end{equation}
where, abusing notation, we identify $\Gamma$ with the union of the supports of $\gamma_i$.
We define the minimal energy induced by the pair  $(\Gamma,\mathbf{B})$ as
\begin{equation} \label{energia4}
E_{\alpha, r} ( \Gamma, \mathbf{B}) := \inf \left\{ \int_{\Omega_r^+} W(F(x)) \, dx \, \colon F \in \AS  \right\} 
\end{equation}
and the minimal energy induced by the lattice misfit as
\begin{equation}\label{minen}
E_{\alpha,r}:= \min \left\{ E_{\alpha,r}(\Gamma,\mathbf B) \, \colon \, (\Gamma,\mathbf B)\in \AD  \right\}.
\end{equation}
Notice that, by the growth assumptions \eqref{growth1} on $W$, the minimum problem in \eqref{minen} involves only dislocations with Burgers vectors in a finite set,
so that the existence of a minimizer is trivial. 
We denote by $E_{\alpha, r}(\emptyset)$ the minimal elastic energy induced by curl free strains.
Notice that $E_{\alpha, r}(\Gamma,\mathbf{B})= E_{\alpha, r}(\emptyset)$ whenever $\Gamma \cap S_r = \emptyset$.\\

{\bf Notation.} Throughout the paper the same letter $C$ denotes various positive constants that are independent of $r$, but whose precise value may change from place to place.

\subsection{Scaling properties of the energies}
The next proposition, proved in  \cite[Proposition 3.2]{mp}, states that the quantities defined by 
\eqref{energia4} and \eqref{minen} are strictly positive. 
\begin{proposition} \label{prop:positivo}
For all $r>0$ one has  $E_{\alpha, r}>0$. Moreover,  $E_{\alpha, r}(\emptyset)=r^3 \, E_{\alpha, 1}(\emptyset)$, 
with  $E_{\alpha, 1}(\emptyset)>0$.
\end{proposition}

Proposition \ref{prop:positivo} asserts that $E_{\alpha, r}(\emptyset)$ grows cubically in $r$. We will show that the energy \eqref{energia4} can grow quadratically in $r$ by suitably introducing dislocations on $S_r$.
In fact we will introduce  dislocations on many (of the order of ${(r(\alpha-1)/b)}^2$)  square  curves.


\begin{proposition} \label{prop:dislo}
There exists  $0<E_{\alpha}<+\infty$ such that
\begin{equation}\label{liminfsu}
\lim_{r\to +\infty } \frac{E_{\alpha,r}}{r^2}= E_{\alpha}.
\end{equation}

\end{proposition}

\begin{proof} 
For the sake of computational simplicity, we assume that \eqref{simp} holds, so that $r/2\in b\N$ (see Remark \ref{rem-qualsiasi}). 
We first show that the limit exists. Let $m,\, n\in \N$ with $n>m$, and let $j$ be the integer part of $\frac{n}{m}$, $R:= nb$, $r:=mb$.  Then, there are $j^2$ disjoint squares of size $r$ in $S_{R}$, so there are $j^2$ disjoint sets equivalent to $\Om_r$ (up to horizontal translations) in $\Om_R$.
By minimality, $E_{\alpha,r}$ is smaller than the energy stored in each of such domains, so that 
$$
\frac{{E_{\alpha,r}}}{r^2} \le \frac{{E_{\alpha,R}}}{r^2j^2} = \frac{{E_{\alpha,R}}}{R^2 + q(r)}\, ,  
$$
where $q(r):= -\big[ (\frac{R}{r} - j)^2  +2j (\frac{R}{r} - j) \big] r^2 = o(R^2)$.
Since this inequality holds true for all $r,R\in b\N$ with $r \leq R$, we deduce that  
$$
\liminf_{n\to +\infty } \frac{E_{\alpha,bn}}{(bn)^2}= \limsup_{n\to +\infty } \frac{E_{\alpha,bn}}{(bn)^2} =\lim_{n\to +\infty } \frac{E_{\alpha,bn}}{(bn)^2}  = :E_{\alpha}.
$$

In order to establish that $E_{\alpha}>0$, it suffices to recall Proposition \ref{prop:positivo} and observe that $E_{\alpha,1}>0$.

Next we show that $E_{\alpha}< +\infty$.
For this purpose, we will exhibit a sequence of deformations and associated dislocations for which the energy grows at most quadratically in $r$.
The construction uses some ideas introduced in \cite{mp08} and \cite{mp}.
Let $\delta := \frac{b}{(\alpha-1)}$ and  recall that by \eqref{simp} we have $r/\delta \in \N$. 
Denote by 
$Q_i$, $i=1 , \dots, q$, the squares of side $\delta$ with vertices in the lattice $S_r \cap \delta \Z^2$,  
and let $x_i$ be
the center of each $Q_i$.
Since the side of $S_r$ is $r$, we have that
$q={(r/\delta)}^2$.

We will define a deformation $v \colon \Omega_r \to \rtre$ such that
$v=x$ in $\Omega^-_r$, $v=\alpha x$ if $x_3 >\delta$ and  the transition from $x$ to $\alpha x$ is 
distributed into constant jumps across the squares $Q_i$'s. In this way the energy will be concentrated 
in a $\delta$-neighbourhood of the interface $S_r$ and the contribution to the energy will come mostly from dislocations.

To this end, let $C_i^1$ and $C_i^2$ be the pyramids of base $Q_i$ and vertices
$x_i + \delta/2 \, e_3$ and $x_i+\delta e_3$ respectively.
Define a displacement $u \colon \Omega_r \to \rtre$ such that
 \[
 u(x)= \begin{cases}
 (\alpha-1) x   & \text{ if } x \in \Omega^+_r \smallsetminus \cup_{i=1}^q C_i^2  \\
 0                     & \text{ if } x \in \Omega^-_r \,.
 \end{cases}
 \]
We complete the above definition by setting $u:=u_i$ in $C^2_i$, where $u_i$ is the unique solution of 
the minimum problem
\begin{equation} \label{min prob pir}
\begin{aligned}
m_{\delta,p, (\alpha-1)I} := \min \left\{ \int_{C^2_i} \va{\nabla v}^p \, \colon \, v\in W^{1,p}_{loc}(\rtre_+), \,  v\equiv (\alpha-1)x_i \text{ in } C^1_i, \right. \\ 
\left. \phantom{\int_{C^2_i} }
v(x)=(\alpha-1) x \text{ in } \rtre_+\setminus C^2_i \right\},
\end{aligned}
\end{equation}
where $\rtre_+:= \rtre\cap\{x_3>0\}$.
Notice that $m_{\delta,p, (\alpha-1)I}$ is independent of $i$ and that $u$ is well defined; 
indeed if $Q_i$ and $Q_j$ are adjacent squares, i.e. 
\[
Q_j=Q_i \mp \delta e_s \qquad \text{for some} \quad s \in \{1,2\},
\]
then
\[
u_j(x)= u_i (x \pm \delta e_s) \mp (\alpha-1) \delta e_s   \quad \text{for every} \quad x \in Q_j \times [0,+\infty].
\]  
Moreover, 
in Proposition \ref{prop:doppio cilindro} we will show  that $0<m_{\delta,p, (\alpha-1)I}<+\infty$ and
 \begin{equation} \label{doppia pir}
 m_{\delta,p, (\alpha-1)I} = \delta^3 {(\alpha-1)}^p \, m_{1,p, I} \, .
 \end{equation}


Set $v=x+u$.
Notice that the deformation $v$ has constant jump equal to $(\alpha-1)x_i$ across $Q_i$. Therefore, if $Q_i$ and $Q_j$ are adjacent and we set $\gamma_{ij}:=Q_i \cap Q_j$, we  have
that $\gamma_{ij}$ is a dislocation line with Burgers vector $\mathbf{b}_{ij}=(\alpha-1) (x_j-x_i)$ (see Figure \ref{fig:burgers}). 
By construction $\gamma_{i,j}$ lies in the grid $(b\Z \times \R) \cup (\R\times b\Z)$. Moreover, since  $\delta=b/(\alpha-1)$,
$\mathbf{b}_{ij} \in \pm b \{e_1,e_2\}$. 
Therefore, setting 
$\Gamma:=\{\gamma_{ij}\}$ and $\mathbf{B}:=\{\mathbf{b}_{ij}\}$, we  
have that $(\Gamma,\mathbf{B}) \in \AD$ and $v \in \AS$. 

We are left to estimate from below the elastic energy of $v$. 
\begin{figure} 
\centering   
\def\svgwidth{9cm}   
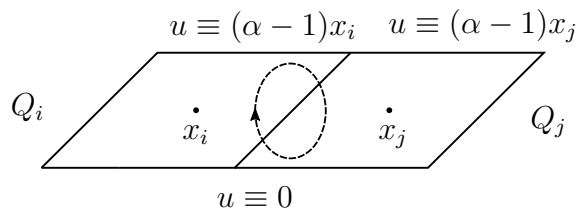  
\caption{A loop winding around $\gamma_{ij}=Q_i \cap Q_j$.}   
\label{fig:burgers}
\end{figure} 
Recalling that
$W(\alpha I)=0$,
the growth condition \eqref{growth1} and \eqref{doppia pir}, we get 

\[
\begin{aligned}
\int_{\Omega^+_r} W(\nabla v) & =\sum_{i=1}^q \int_{C_i^2} W(\nabla v) \leq C \sum_{i=1}^q \int_{C_i^2} (\va{\nabla v}^p +1) \\ 
& \leq C q \va{C_i^2} + q \delta^3 {(\alpha-1)}^p \, m_{1,p,I} =
 q \delta^3 \left( C+{(\alpha-1)}^p \, m_{1,p,I} \right)  \,.
\end{aligned}
\]
Writing $q=r^2/\delta^2$ and $\delta=b/(\alpha-1)$ yields



%

\begin{equation} \label{tesi11}
\int_{\Omega^+_r} W(\nabla v) \leq  r^2 b \left[ {(\alpha-1)}^{p-1} \, m_{1,p,I}  + {(\alpha-1)}^{-1} C \right]\,.
\end{equation}
\end{proof}
\begin{remark}\label{rem-qualsiasi}
{\rm
In the case when \eqref{simp} does not hold, it suffices to observe that
$$
E_{\alpha,[\frac{r}{2\delta}]2\delta} \leq E_{\alpha,r} \leq E_{\alpha,[\frac{r}{2\delta}]2\delta+2\delta}\quad \text{ and }\quad 
\lim_{r\to \infty}
\frac{([\frac{r}{2\delta}]2\delta)^2}{r^2}= 
\lim_{r\to \infty}
\frac{([\frac{r}{2\delta}]2\delta + 2\delta)^2}{r^2}
=1,
$$
where $[a]$ denotes the integer part of $a$.
The above inequalities follow from the fact that if $r_1<r_2$, then the restriction to $\Om_{r_1}$ of any test function 
for $E_{\alpha,r_2}$ provides a test function for 
$E_{\alpha,r_1}$.
}
\end{remark}

\begin{remark}{\rm
The asymptotic behavior of the energy described by Proposition \ref{prop:dislo} strongly relies 
on the structure assumption made on the admissible dislocation lines. 
Proving lower bounds on the energy can in general be a delicate problem if no structure assumption is made. 
In fact, local estimates of the energy can be obtained in a neighborhood  of the dislocation lines, 
as long as these are sufficiently regular and well separated.
} 
\end{remark}

As a corollary of Propositions \ref{prop:positivo} and \ref{prop:dislo} we obtain the following theorem, asserting that 
nucleation of dislocations is energetically convenient for sufficiently large values of $r$.

\begin{theorem} \label{thm:dislo favorevoli}
There exists a threshold $r^*$ such that, for every $r > r^*$, 
\[
E_{\alpha,r}  < E_{\alpha,r} (\emptyset).
\] 

\end{theorem}



\subsection{Double pyramid construction}\label{piramidi}

Fix $\delta>0$ and let  $C^1$ and $C^2$
be the pyramids with common base the square $(-\delta/2,\delta/2)^2\times \{0\}$
and heights $\delta/2$ and $\delta$ respectively.
Note that $C^1 \subset C^2$. 
Set $S:=C^2 \cap \{\delta/2 < x_3 < \delta \}$ and
 $T:=(C^2 \smallsetminus C^1) \cap \{0 < x_3 < \delta/2\}$. See Figure \ref{fig:doppio cono} for a cross section of the construction in cylindrical coordinates.

\begin{figure} 
\centering   
\def\svgwidth{3.5cm}   
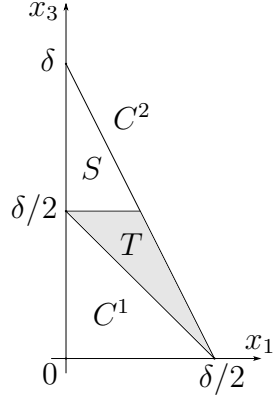  
\caption{Section $\f=0$ of the double pyramid.}   
\label{fig:doppio cono}
\end{figure} 

Let $A\in \matrici$ with $A\neq 0$, and consider the following minimization problem
\begin{equation}\label{mdpa}
m_{\delta, p, A}:= \min \left\{ \int_{C^2} \va{\nabla v}^p \, dx \, \colon \, v\in W^{1,p}_{loc}(\rtre_+), \,   v\equiv 0 \text{ in } C^1, \,  v\equiv A x \text{ in } \rtre_+\setminus C^2 \right\}, 
\end{equation}
where $\rtre_+:= \rtre\cap\{x_3>0\}$.

\begin{proposition} \label{prop:doppio cilindro}
The minimum problem \eqref{mdpa} is well posed. Moreover,
\begin{itemize} 
\item[i)] $0<m_{1, p, A}<+\infty$ for every $1<p<2$;
\item[ii)] 
$m_{1, 2, A} = \infty$;
\item[iii)] for all positive $\delta$ and $\lambda$ we have
$m_{\delta, p, \lambda A} = \delta^3 \lambda^p  m_{1, p, A}$.
%

 \end{itemize}
\end{proposition}

\begin{proof}
The fact that the minimum problem is well posed is standard and based on the direct method of the calculus of variations. 

Property iii$)$ holds because if $v$ is a competitor for $m_{\delta,p, \lambda A}$, then
$\tilde{v}(x):= v(\delta x)/ \lambda \delta$ is a competitor for $m_{1,p,A}$.

As far as i) is concerned, first remark that $m_{1, p, A}>0$. Indeed, arguing by contradiction assume that  
$m_{1, p, A}=0$. Then, the minimizer $v$ would satisfy $\nabla{v} \equiv 0$ in $C^2$ and $\nabla{v} \equiv A$ in $\rtre_+ \setminus C^2$, which is 
a contradiction since this is only possible when $A =  0$. 

We will prove that $m_{1, p, A}<+\infty$ by exhibiting an admissible deformation $v$ with finite energy. 
In order to simplify the computations, we will show it in the case when $C^1$ and $C^2$ are the cones with base the disk of diameter $1$ and center the origin, and heights $1/2$ and $1$ respectively.
The estimate in the case of two pyramids can be proved in the same way, with minor changes.

Introduce the cylindrical coordinates 
$x_1= \rho \cos \f$, $
x_2= \rho \sin \f$ and $x_3=z$,
with $\rho >0$ and $\f \in [0,2 \pi)$. Set $v:= 0$ in $C^1$ and
$v(x):=Ax$ in $\rtre_+ \setminus C^2$. 
First we extend $v$ to $S$.
To this end, for all $\bar\f \in [0,2 \pi)$ we define $v$ 
in the triangle  $S_{\bar\f}:= S\cap\{\f=\bar\f\}$ by linear interpolation of the values of $v$ at the three vertices of $S_{\bar\f}$.  
Notice that $v$ is Lipschitz continuous in $S$. 
Next, we extend $v$  to $T:=C^2\setminus (S\cup C^1)$. 
For this purpose, for all $\bar\f \in [0,2 \pi)$ and $\bar z\in(0,\frac12)$  consider the segment  $L_{\bar\f,\bar z}:= T\cap \{\f=\bar \f\}\cap\{z=\bar z\}$, and  define $v$ on  
$L_{\bar\f,\bar z}$ by linear interpolation of the values of $v$ on the two extreme points of  $L_{\bar\f,\bar z}$.

We will now estimate the $L^p$ norm of $\nabla v$ in $C^2$. Since
 $v$ is piecewise Lipschitz in $C^2 \setminus T$, we only have to compute the energy in $T$.
By construction we have that 
\begin{equation}\label{stime} 
 |\nabla v(x,y,z)| \le \frac{c}{ z} \qquad \text{ for all } (x,y,z) \in T,
\end{equation}
where $c$ are suitable positive constant depending only on $A$. 
A straightforward computation yields $m_{1,p,A} \leq C(p,A)$ with the constant $C$ depending only on $A$ and $p$, and diverging as $p \to 2^{-}$, 
which proves i).

Finally, let us prove that $m_{1,2,A}=+ \infty$. 
For every admissible function $v$ and all  $0< \e  < 1/2$, by Jensen's inequality we have
\begin{equation*} 
\int_{T \cap \{ \e < z < \frac12     \}} |\nabla v|^2 \, dx \ge   
\int_{T \cap \{ \e < z < \frac12     \}} \va{ \frac{\partial v}{\partial \rho}}^2 \, dx \ge
c \int_{\e}^{\frac12} \frac{1}{s} {\left( \int_{T\cap \{z=s\}} \frac{\partial v}{\partial \rho} \, d \rho   \right)  }^{2} \, ds \geq 
c \log \frac{1}{\e} \,.
\end{equation*}

Taking the limit $\e \to 0$ yields 
$\int_{C^2} \va{\nabla v}^2 = + \infty$.
\end{proof}


\subsection{Some considerations on the proposed model}\label{conjr}

In the construction illustrated in the proof of Proposition \ref{prop:dislo},
$v(S_r)$ is union of disjoint squares of size $\delta$, separated by strips of width $b$; dislocation lines lie in the middle of such strips (see Fig. \ref{lungh}). 
Note that some lines of atoms (in the deformed configuration) fall out of   $S_r$ (see Figure \ref{nonf}), suggesting that 
 the chosen reference configuration is not convenient to describe heterostructured nanowires, 
 or epitaxial growth.
\begin{figure}[h]
\begin{center}
{\includegraphics[scale=0.35]{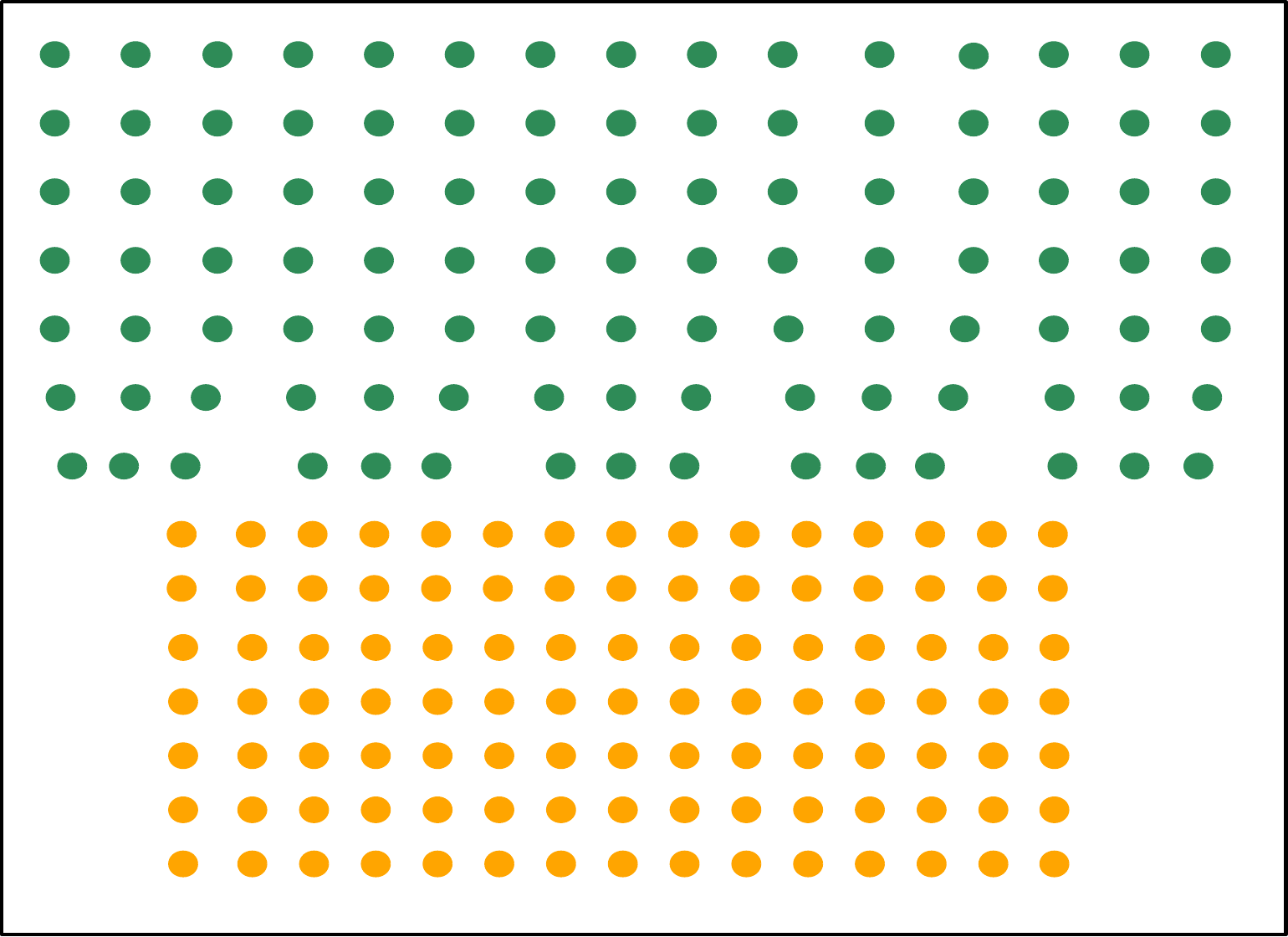}}
\caption{\label{nonf} A 2D discrete cross section of the crystal, deformed 
according to the construction made in the proof of Proposition \ref{prop:dislo}.}   
\end{center}
\end{figure}

In fact, this is not the physical configuration we are interested in modeling and analyzing: The semi-discrete model presented in this section is somehow  
meant as a theoretical background to derive material constants, and in particular the energy per unit dislocation length and interface area, that will be 
involved in the model discussed in Section \ref{contmod}.



In order to avoid configurations like in Figure \ref{nonf}, in the next section we will rather modify our point of view, dealing with a reference configuration 
$\Omega_{R,r}:=\Omega_{R}^- \cup S_{r} \cup \Omega_{r}^+$ with $r:= \theta R$ for some $0<\theta<1$ (see Figure \ref{fig:ref2}), and enforcing that $v(S_r) =S_R$, thus describing  
a perfect mach between the two parts of the crystal as in Figure \ref{freedisloc}. The new parameter
$\theta$ represents the  ratio between the size of $S_r$ and that of its deformed counterpart $v(S_r)$. 
Optimization over $\theta$ corresponds to ``getting rid'' of unnecessary atoms at the interface 
and will yield $\theta \approx  \alpha^{-1}$ in the limit $r\to \infty$.

In this context it is quite natural to measure the dislocation length in the deformed configuration $v(\Omega_{r}^+)$. 
In the construction made in the proof of Proposition \ref{prop:dislo},
the number of  dislocation straight-lines  is of the order $\frac{2r}{\delta} $, where $\delta=\frac{b}{\alpha -1}$.
Mimicking the same construction in the new reference configuration $\Omega_{R,r}$, in order to enforce $v(S_r) =S_R$, now we have to choose $\delta=\frac{b}{\theta^{-1} -1}$.
The total length $L$ of dislocations (in the deformed configuration) is then of the order $L= \frac{2r^2} { b} (\theta^{-2} -\theta^{-1})$. 
The above formula can be obtained alternatively as follows. Let  $\tilde L$ be the total length of dislocations in the reference configuration. Then, $b \tilde L$  corresponds to  the 
total variation of the curl of the deformation strain in  $\Omega_{r}^+$. Since the jump of the strain across the interface is $(\theta^{-1}-1) I$, 
the modulus of  the curl  is equal to $2 (\theta^{-1}-1)$.  We deduce
$$
L = \theta^{-1} \tilde L = \frac{\theta^{-1}}{b} r^2 2 \left(\theta^{-1}-1\right) = \frac{2r^2} { b} (\theta^{-2} -\theta^{-1}).
$$ 

We are interested in small misfits $\theta^{-1} \approx 1$. Therefore,  $(\theta^{-2} -\theta^{-1})\approx \frac12 (\theta^{-2} -1)$, so that 
the  total length of dislocations is  of the order 
$$
L= \frac{1} { b} r^2(\theta^{-2} -1) = \frac{1}{b} \text{ Area Gap},
$$
where 
Area Gap, in a continuous modeling of the crystal, represents the difference between the area of the base of the deformed configuration $ v(\Om^+_r) = S_R$ of $\Om^+_r$, and the area of the base of the reference configuration, namely the area of $S_r$ (see Figure \ref{lungh} and Figure \ref{fig:ref2}).  

\begin{figure}[h]
\begin{center}
{\includegraphics[scale=0.40]{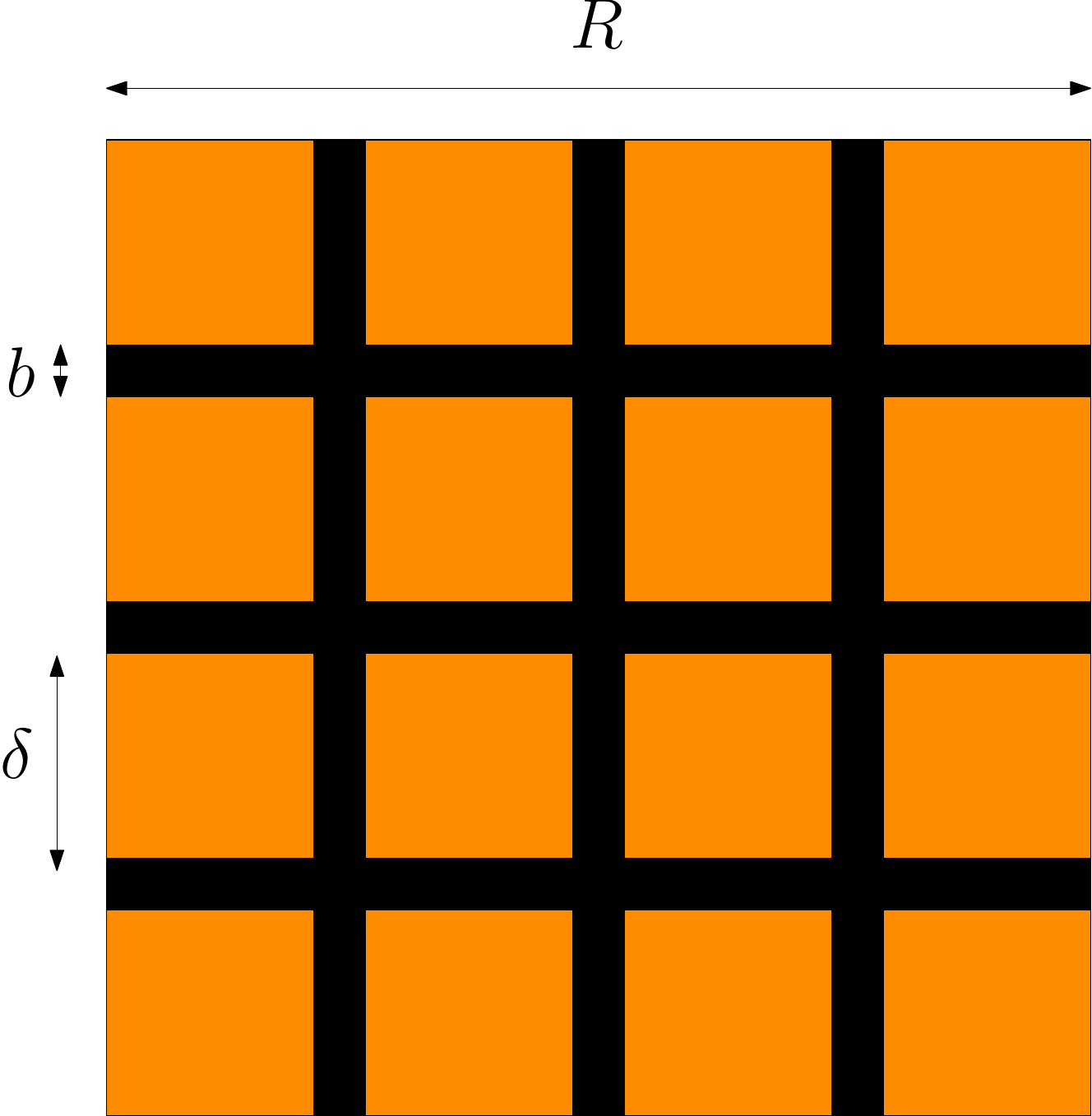}}
\caption{\label{lungh} The deformed configuration $v(S_r)$. The black strips represent dislocations. The total dislocation length is given, in first approximation, by the area of the black region divided by $b$.}   
\end{center}
\end{figure}

We do not claim that our constructions are optimal in energy. Nevertheless, we believe that, as $r\to \infty$, 
 the optimal configuration of dislocations exhibit some periodicity. As a matter of fact, in Proposition \ref{prop:dislo} we have proved  that 
\begin{equation}\label{conj0}
E_{\alpha,r} \approx \sigma_{\alpha,\theta}  \text{ Area Gap} \qquad \text{ as } r\to +\infty,
\end{equation}
for  
$$
\sigma_{\alpha,\theta} = \frac{E_{\alpha}}{\theta^{-2}-1}. 
$$
 In view of the considerations above, this reflects that the energy is proportional  to the total dislocation length.  
In particular,  as $r\to \infty$ and  $\alpha \to 1^+$, we expect that  $E_{\alpha,r}$ is minimized by a periodic configuration of dilute and well separated dislocations. 
Taking this into account, we expect  that  
\begin{equation}\label{conj}
\lim_{ \alpha \to 1^+}  \frac{E_{\alpha}}{\alpha^2-1} = \lim_{ \alpha \to 1^+}  \sigma_{\alpha,\alpha^{-1}} =:\sigma,
\end{equation}
for some $0<\sigma<\infty$, where $b \sigma $ represents the self energy of a single dislocation line per unit length.

Let us compare the nonlinear energy induced by dislocations with the solid framework of linearized elasticity. It is well known that the energy per unit (edge dislocation) length  in a single crystal of size $r$ is given by $b^2 \frac{\mu}{4\pi (1-\nu)} \ln(\frac{r}{b})$ (see, e.g., \cite{HL,NAB}), where $\mu$ is the shear modulus and $\nu$ is the Poisson's ratio. Here, $r$ should be replaced by the average distance between dislocations, and an extra prefactor 2 appears due to the fact that $\Om^-_R$ is rigid, so that the Burgers circuit is indeed half circuit. The resulting energy per unit dislocation length is then   
\begin{equation}\label{Endislo}
\gamma^{lin}:= b^2 \frac{\mu}{2\pi (1-\nu)} \ln\left(\frac{1}{\alpha-1}\right).
\end{equation}
To such energy, a chemical core energy $\gamma^{ch}$ per unit dislocation length should be added. Notice that this contribution is already present in our nonlinear formulation, 
and it is stored in the region where $|\nabla v|$ is large, and the energy density $W(\nabla v)$ behaves like $|\nabla v|^p$. We deduce  that, for small misfits,
$$
(\gamma^{lin} + \gamma^{ch}) \frac{1}{b} \text{ Area Gap } \approx E_{\alpha,r} \approx \sigma  \text{ Area Gap}, 
$$
which yields the following expression for $\sigma$:
\begin{equation}\label{sigma?}
\sigma = b \frac{\mu}{2\pi(1-\nu)} \ln \left(\frac{1}{\alpha-1}\right) + \frac{1}{b} \gamma^{ch}.
\end{equation}
Finally, we notice that $\sigma(\alpha^2-1)$ is noting but the energy per unit surface area, so that the total energy is given by
$$
 E_{\alpha,r} \approx  r^2 (\alpha^2-1)\left( b \frac{ \mu}{2\pi(1-\nu)} \ln \left(\frac{1}{\alpha-1}\right) + \frac{1}{b} \gamma^{ch}\right).
$$ 
\section{A simplified continuous model for dislocations}\label{contmod}
Based on the analysis and the considerations on the semi-discrete model discussed in Section \ref{mod1} (see Subsection \ref{conjr}), here we want to propose a simplified and more realistic model for dislocations at interfaces. Instead of working with $\rm SBV$ functions with piece-wise constant jumps on the interface, we will allow only for regular jumps but we will introduce a penalization to the elastic energy, that will represent the dislocation energy.
\subsection{The simplified energy functional}\label{wdoppio}
Fix $\alpha>1$, $R>0$, $\theta \in [\alpha^{-1},1]$ and set $r:=\theta R$.  Let 
$\Omega^-_R:=S_R \times (-hR,0)$, where 
$S_R \subset \rdue$ is the square of side length $R$ centered at the origin and $h>0$ a fixed height. 
Define now the reference configuration (see Figure \ref{fig:ref2}),
$$
\Omega_{R,r}:=\Omega_{R}^- \cup S_{r} \cup \Omega_{r}^+.
$$
 \begin{figure}[h] 
\centering   
\def\svgwidth{6cm}   
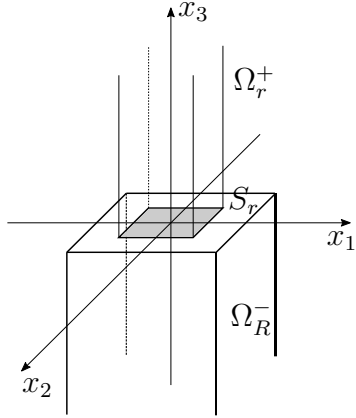  
\caption{Reference configuration $\Omega_{R,r}$\,.}   
\label{fig:ref2}
\end{figure} 
As in Section \ref{mod1} we will suppose that $\Omega^-_R$ is rigid 
 and that $\Omega^+_r$ is in equilibrium with $\alpha I$. We assume that there exists an energy density $W \colon \matrici \to [0,+\infty)$ that is continuous, $C^2$ in a neighborhood of $SO(3)$ and frame indifferent (see \eqref{frame ind}). Furthermore we suppose that 
 \begin{equation} \label{nulla}
W(\alpha I) =0
\end{equation}
and that for every $F \in \matrici$
\begin{equation} \label{growth2}
C_1   \,   \dist^2 (F, \alpha \rotazioni )  \leq W(F)  
\end{equation}
for some constant $C_1>0$.
The class of admissible deformation maps is defined by 
\begin{equation}
\ADM_{\theta,R} :=\left\{v\in W^{1,2}(\Omega^+_{r}; \rtre) \, \colon \,  v (x) = \frac{1}{\theta} \, x \, \text{ on }  S_{ r} \right\}.
\end{equation}
In this way $v(S_{ r})=S_R$ for all $v\in \ADM_{\theta, R}$. 
A deformation $v \in \ADM_{\theta,R}$ stores an elastic energy 
\[
E^{el}_{\alpha,R}(\theta,v) := \int_{\Omega_{r}^+ }  W (\nabla v (x)) \, dx \,  .
\] 
To this energy we add
a dislocation energy $E^{pl}_{R}(\theta)$ proportional
to the area of $S_R \smallsetminus S_{ r}$, which in our model represents the total dislocation length, 
\begin{equation}\label{eplastic}
E^{pl}_{R}(\theta):= \sigma r^2 ({\theta}^{-2}-1) = \sigma R^2 (1-{\theta}^2).
\end{equation}
Here $\sigma >0$ is a given constant, which in our model is a material property of the crystal, representing (multiplied by $b$) the energetic cost of dislocations per unit length.  
In principle, $\sigma$ could be derived starting from the semi-discrete model discussed in Section~\ref{mod1}: since we will see that the optimal $\theta$ tends to 
$\frac{1}{\alpha}$ as $R\to\infty$, and we are interested in small misfits $\alpha\approx 1$, it is consistent to choose $\sigma$ according to \eqref{conj}, so that $b \sigma$ represents the energy induced by a single dislocation line per unit length. 
We are thus led to study the energy functional 
\[
\tot (\theta,v):= E^{el}_{\alpha,R}(\theta,v)  + E^{pl}_{R}(\theta) =  \int_{\Omega_{r}^+ }  W (\nabla v (x)) \, dx + \sigma R^2 (1-{\theta}^2) 
\]
We further define
\begin{equation}
\el (\theta):= \inf \left\{E^{el}_{\alpha,R}(\theta,v)  \, \colon \,   v \in \ADM_{\theta, R} \right\}  , \qquad \tot (\theta) := \el (\theta) + E^{pl}_{R}(\theta).
\end{equation}
Notice that if $\theta=1$ then no dislocation energy is present, i.e. $\tot (1)= \el (1)$. Instead, if 
$\theta= \alpha^{-1}$ no elastic energy is stored (since $v(x):= \alpha x$ is admissible and $W(\alpha I)=0$) (see Figure \ref{freedisloc}). 

\begin{proposition} \label{prop:cubic}
The elastic energy $\el (\theta)$ satisfies
\begin{itemize}
\item[i)] $\el(\theta) = R^3 \theta^3  E^{el}_{\alpha,1}(\theta)$
\item[ii)] $E^{el}_{\alpha,1}(\theta)>0$ if and only if $\theta > \alpha^{-1}$.
\end{itemize}
\end{proposition}

\begin{proof}
Property i$)$ follows by noticing
that if $v$ is in $\ADM_{\theta,R}$, then $\tilde{v}(x):= v(R \theta x)/R \theta$ is in $\ADM_{\theta,1}$. For the second property,
we have to prove that $E^{el}_{\alpha,1}(\theta)=0$ if and only if $\theta = \alpha^{-1}$.
We already pointed out that $E^{el}_{\alpha,1}(\alpha^{-1})=0$. 
Suppose that $E^{el}_{\alpha,1}(\theta)=0$. Then there exists a sequence $v_n \in H^1(\Omega_1^+;\rtre)$ such that $\su{v_n}{S_1}=\theta^{-1} \, x$ and
\begin{equation} \label{limite}
\int_{\Omega_1^+} W(\nabla v_n) \, dx \to 0 \quad \text{ as }  n \to \infty \,.
\end{equation}
The Rigidity
Theorem \ref{thm:rig}, the growth assumption \eqref{growth2} and the compactness of $\rotazioni$ 
in combination with \eqref{limite} imply that there exists a fixed rotation
$R \in \rotazioni$ such that (up to subsequences)
\begin{equation*} 
\int_{\Omega_1^+} \va{\nabla v_n - \alpha R}^2 \, dx \to 0 \quad \text{ as }  n \to \infty \,.
\end{equation*}
Setting
$\zeta_n := (1/\va{\Omega_1^+}) \int_{\Omega_1^+} (v_n(x) - \alpha R x)$,
from the Poincar\'e  inequality and the trace theorem we deduce that
\begin{equation} \label{limite6}
\int_{S_1}\va{v_n - \alpha R x - \zeta_n}^2 \, dx \to 0 \quad \text{ as }  n \to \infty \,.
\end{equation}
Since $v_n = \theta^{-1}x$ on $S_1$, \eqref{limite6} yields 
\begin{equation}\label{limite7}
 (\theta^{-1}I - \alpha R)x - \zeta_n \to 0 \quad \text{ in } L^2(S_1) \,,
\end{equation}
which implies $\zeta_n \to 0$, $R=I$ and $\theta= \alpha^{-1}$.
 \end{proof}

In analogy with Theorem \ref{thm:dislo favorevoli},  we find that 
for $R$ sufficiently large dislocated configurations are energetically preferred.

\begin{theorem} \label{thm:dislo favorevoli2}
There exists a threshold $R^*$ such that, for every $R > R^*$
\begin{equation} \label{fav}
\inf_{\theta \in [\alpha^{-1},1)} \tot (\theta) <  \tot (1)	=\el(1) \, .
\end{equation}
\end{theorem}
\begin{proof}
The left hand side of \eqref{fav} can grow at most quadratically in $R$, indeed
\begin{equation*}
\inf_{\theta \in [\alpha^{-1},1)} \tot (\theta) \leq \tot (\alpha^{-1}) =   \sigma R^2 \left( 1- \frac{1}{\alpha^2} \right) \,.
\end{equation*}
In contrast, by Proposition \ref{prop:cubic}, the right hand side $E^{tot}_{\alpha,R}(1)$ grows cubically in $R$.
\end{proof}
%

%

Theorem \ref{thm:dislo favorevoli2} motivates the following definition
\begin{equation} \label{ott}
E^{tot}_{\alpha ,R} := \inf_{\theta \in [\alpha^{-1},1]} \tot(\theta) \,. 
\end{equation}
One can show that $\tot$ is continuous, so that the infimum is in fact a minimum.  
Our goal is to study the asymptotic behavior  of  $\tot$ as $R \to \infty$. 
In Theorem \ref{il teorema} we will write $\tot$ as an expansion in powers of $R$.

\subsection{An overview of the Rigidity Estimate and Linearization}

We recall the Rigidity Estimate from \cite{fjm}. In this section, $U \subset \rtre$ will be a Lipschitz bounded domain.

\begin{theorem}[Rigidity Estimate, \cite{fjm}] \label{thm:rig}
There
exists a constant $C>0$ depending only on the domain $U$ such that for every $v \in H^1(U;\runo^3)$ there exists a constant rotation
$R \in \rotazioni$ such that
\begin{equation} \label{rig}
\int_{U} \va{ \nabla v(x) - R}^2 \, dx \leq C \int_{U} \dist^2 (\nabla v(x) ;  \rotazioni) \, dx \,.
\end{equation}
\end{theorem}

In order to compute the Taylor expansion of $\tot$ defined in \eqref{ott}, we will linearize the elastic energy as in \cite{dnp}. Therefore, following \cite{dnp}, we will  make further assumptions on $W$.
First, notice that by minimality the equilibrium $\alpha I$ is stress free, i.e.
\begin{equation}
\partial_F W (\alpha I) = 0 \,. \label{stressfree}
\end{equation}
By frame indifference there exists a function $V \colon \simm \to [0,+\infty]$,  such that
\begin{equation} \label{relazione}
W(F) = V\left( \frac{1}{2} \left( F^T F - \alpha^2 I \right) \right)  \qquad \text{for every} \quad F \in \matrici_+.
\end{equation}
Here $\simm$ is the set of $3 \times 3$ symmetric matrices, $\matrici_+$ is the subset of matrices with positive determinant  and $F^T$ is the transpose matrix of $F$. 

The regularity assumptions on $W$ (see Subsection \ref{wdoppio}) imply that $V(E)$ is of class $C^2$ in a neighbourhood of $E=0$.
From \eqref{nulla}, \eqref{stressfree} and \eqref{relazione} it follows that
$V(0)=0$ and
$\partial_E V (0) = 0$.
Moreover, by \eqref{growth2}, there exist $\gamma,\, \delta>0$ such that  
\[
\partial^2_E V (E)[T,T] \geq  \gamma \va{T}^2 \quad \text{ if } \va{E} < \delta \text{ and } T \in \simm.
\]
By Taylor expansion we find
\begin{equation} \label{taylor V}
V(E)=\frac{1}{2} \partial^2_E V(0)[E,E]+o(|E|^2) \,.
\end{equation}

Let $v \in  W^{1, \infty}(U;\runo^3)$ with $\nabla v\in \matrici_+$, and write 
$v=\alpha x + \e u$. 
Then from \eqref{relazione} it follows
\[
W(\nabla v)= V \left( \alpha \e e(u) + \frac{\e^2}{2} C(u) \right),
\]
where $e(u):= (\nabla u + \nabla u^T )/2$ and 
$C(u):=\nabla u^T \nabla u$.
By \eqref{taylor V} we get
\begin{equation} \label{taylor3}
W(\nabla v)=  \, \frac{\e^2}{2} A [e(u),e(u)] + o(\e^2) 
\end{equation}
where
$A := \alpha^2 \partial^2_E V(0)$ is the stress tensor.
Notice that \eqref{taylor3} is uniform in $x$, since $\nabla v$ is bounded.
In particular
\begin{equation*} 
\lim_{\e \to 0} \, \frac{1}{\e^2} \int_{U} W(\alpha I + \e \nabla u) \, dx = \frac{1}{2} \int_{U} A[e(u),e(u)] \, dx \,.
\end{equation*}
In \cite{dnp} it is proved that the above convergence holds also for minimizers. 
Specifically, let $\Sigma \subset \partial U$ be closed and such 
that $\mathcal{H}^2 (\Sigma) >0$. 
Introduce the space
\[
H^1_{x,\Sigma} (U;\rtre) := \left\{ u \in H^1(U;\rtre) \, \colon \, u(x)=x \text{ on } \Sigma \right\} 
\]
and, for $u\in H^1_{x,\Sigma} (U;\runo^3)$, define the functionals 
\[
G_\e (u) := \frac{1}{\e^2} \int_{U} W(\alpha I + \e \nabla u)
\, dx
\qquad \text{and} \qquad
G(u) := \frac{1}{2} \int_{U} A[e(u),e(u)] \, dx .
\]
We can now recall \cite[Theorem 2.1]{dnp}:
\begin{theorem}[Linearization] \label{thm:dnp}
If $\{u_\e\} \subset H^1_{x,\Sigma} (U;\runo^3)$ is a minimizing sequence, i.e.
\[
\inf_{H^1_{x,\Sigma} (U;\runo^3)} G_{\e}= G_{\e} (u_\e) + o(1)  
\]
then $u_\e$ converges weakly to the unique solution $u_0$ of
\[
\min_{H^1_{x,\Sigma} (U;\runo^3)} G \,.
\]
Moreover we have
\begin{equation} \label{convminimi}
\inf_{H^1_{x,\Sigma} (U;\runo^3)} G_{\e} \to \min_{H^1_{x,\Sigma} (U;\runo^3)} G \quad \text{ as } \quad \e \to 0 \,.
\end{equation}
\end{theorem}

\begin{remark}
{\rm 
In  \cite{dnp} the growth assumption \eqref{growth2} is replaced by 
the milder  condition
\[
W(x,F)=+\infty \qquad \text{if} \qquad \det F \leq 0 \, 
\]
(orientation preserving condition), which
 together with suitable 
growth assumptions on $V$, implies that inequality \eqref{growth2} holds in average. 
}

%
\end{remark}

\subsection{Taylor expansion of the energy}
We can now carry on our analysis. We say that $\theta_R \in [\alpha^{-1},1]$ is
a minimizing sequence for the energy
$\tot$ defined in \eqref{ott} if 
\[
\tot = \tot(\theta_R)+ o(1)
\] 
where $o(1) \to 0$ as $R \to + \infty$. 

\begin{proposition} \label{prop:minimizzante}
Let $\theta_R$ be a minimizing sequence for $\tot$. Then
\begin{itemize}
\item[i)] $E^{el}_{\alpha,1} (\theta_R) \to 0$ as $R \to +\infty$;	
\item[ii)] $\theta_R \to \alpha^{-1}$ as $R \to + \infty$.
\end{itemize}
 \end{proposition}

\begin{proof}
By Proposition \ref{prop:cubic} we have (for $R$ large enough)
\[
R^3 \theta_R^3  E^{el}_{\alpha,1}(\theta_R) = \el (\theta_R) \leq \tot (\theta_R) \leq 
\tot (\alpha^{-1}) + 1 =
 \sigma R^2 \left(1 - \frac{1}{\alpha^2} \right) +1 \,,
\]
which proves i$)$, since
$\theta_R \geq \alpha^{-1}>0$.

Let us now prove ii$)$. 
From i) 
we know that  there exists a sequence $\{v_R\}$ in $H^1(\Omega_1^+; \rtre)$ such that $v_R=\theta_{R}^{-1} \,x$ on $S_1$ and 
\begin{equation} \label{condizione3}
\int_{\Omega_1^+} W(\nabla v_R) \, dx \to 0 \quad \text{ as } \quad R \to  +\infty \,.
\end{equation}

Since $v_R= \theta_R^{-1} \, x$ on $S_1$, the proof is concluded once we show that $v_R \to \alpha x$ in  $H^1(\Omega_1^+;\rtre)$.
This can be shown following the lines of the proof of Proposition \ref{prop:cubic}; the details are left to the reader. 


\end{proof}



If
$v \in H^1(\Omega_1^+;\rtre)$ 
is such that $v=\theta^{-1} \, x$ on $S_1$, then we write
\[
v= \alpha x + \left( \frac{1}{\theta} - \alpha \right) u
\]
where $u \in H^1(\Omega_1^+;\rtre)$ is such that $u=x$ on $S_1$. If we set
$\Sigma=S_1$ we can apply Theorem \ref{thm:dnp} to
the functional $E^{el}_{\alpha,1}(\theta)$ to obtain the following Corollary. 

\begin{corollary} \label{thm:lin}
If $\theta \to \alpha^{-1}$ then
\begin{equation} \label{altro2} 
\frac{1}{{(\theta^{-1}-\alpha)}^2} \, E^{el}_{\alpha,1}(\theta) \longrightarrow C^{el} , 
\end{equation}
where
\[
C^{el} := \min \left\{ \frac{1}{2} \int_{\Omega_1^+} A[e(u),e(u)] \,dx \, \colon \, u \in H^1 (\Omega_1^+;\rtre) , \, u=x \text{ on } S_1 \right\} \,. 
\]
\end{corollary}

From Proposition \ref{prop:minimizzante} we know that 
if $\{\theta_R\}$ is a minimizing sequence, than 
$\theta_R \to \alpha^{-1}$. We can then linearize the elastic energy 
along the sequence $\theta_R$:
\begin{align*}
E^{el}_{\alpha,R}(\theta_R) & = R^3 \theta_R^3 \, E^{el}_{\alpha,1} (\theta_R) 
 = R^3 \theta_R^3 { (\theta_R^{-1}  - \alpha   )   }^2  \, \frac{1}{  { (\theta_R^{-1}  - \alpha   )   }^2} \, E^{el}_{\alpha,1} (\theta_R) \\
& \stackrel{\eqref{altro2}}{=} R^3 \theta_R^3 { (\theta_R^{-1}  - \alpha   )   }^2 (C^{el} + \e_R)
= k^{el}_R R^3 \theta_R \, { \left(\alpha \theta_R -1 \right)}^2 ,
\end{align*}
where $\e_R \to 0$ as $R\to+\infty$ and $k^{el}_R:=C^{el}+\e_R$. 
Since (by Korn's inequality) $C^{el}>0$,  $k^{el}_R>0$ for $R$ sufficiently large (and in fact for all $R$). We are thus led to define the family of polynomials 
\begin{equation} \label{fam pol}
P^{tot}_{k,R} (\theta) :=P^{el}_{k,R} (\theta) + E^{pl}_R(\theta),
\end{equation}
where $k,R>0$ are positive parameters and  $P^{el}_{k,R} (\theta) := k R^3  \theta {(\alpha \theta -1)}^2$. In this way we can write
\begin{equation}\label{finalst}
\tot(\theta_R) =  P^{tot}_{k^{el}_R,R} (\theta_R)  \,.
\end{equation} 
By optimizing $P_{k,R}^{tot}$ with respect to $\theta$, we deduce  the asymptotic behavior of
$\tot$. 
Set
 \[
E_k^{el} (R) :=\frac{\sigma^2}{\alpha^3 k} \, R \qquad \text{ and } \qquad
 E_k^{pl}(R) := \sigma R^2 \left(   1- \frac{1}{\alpha^2}      \right) - 2 \frac{\sigma^2}{\alpha^3 k} R \,.  
\]

\begin{theorem} \label{il teorema}
For every $R, \, k>0$ there exists a unique minimizer $\theta_{k,R}$ of $P^{tot}_{k,R} $ in $[\alpha^{-1},1]$, with   $\theta_{k,R} \to {\alpha}^{-1}$ as $R \to + \infty$.
Moreover,   
\begin{equation}\label{taylorthm}
P^{el}_{k,R} (\theta_{k,R}) =  E_k^{el} (R) +  o(R), \qquad E^{pl}_R(\theta_{k,R}) =  E_k^{pl}(R) + o(R), 
\end{equation}
where $o(R)/R\to 0$ as $R\to +\infty$. In particular, we have
\[
\tot =  E_{C^{el}}^{el} (R) + E^{pl}_{C^{el}}(R) + o(R). 
\]	
\end{theorem}




\begin{proof} We compute the derivative of $P^{tot}_{k,R}$ with respect to $\theta$ 
\[
(P^{tot}_{k,R})' (\theta) = R^2 \left\{   (3 \alpha^2 k R) \theta^2 - 2 (2 \alpha k R + \sigma ) \theta + k R            \right\}.
\] 
One can check  that it vanishes at 
\begin{equation} \label{expr}
\theta_{\pm} (R)= \frac{1}{3 \alpha} \left\{ 2 + \frac{ c}{R}  \pm f(R)  \right\},
\end{equation}
where 
\begin{equation}\label{exprf}
f(R) := \sqrt{ 1+  \frac{4c}{R} + \frac{c^2}{R^2}         } \, \, \,  \quad \text{ and } \quad  c:= \frac{\sigma}{\alpha k } \,.
\end{equation}
Since $f(R) > 1$ we have $\theta_{+}(R) > \alpha^{-1}$. Moreover, 
$f(R) \to 1$, and thus $\theta_{+} (R) \to {\alpha}^{-1}$,  as $R \to +\infty$.
Hence $\theta_{+} (R) \in [\alpha^{-1},1]$ for $R$ large enough. Also note that
$\theta_{-} (R)<\alpha^{-1}$ for $R$ sufficiently large. 
The second derivative is given by
\[
(P^{tot}_{k,R})'' (\theta) = R^2 \left\{ (6 \alpha^2 k R) \theta - 2 (2 \alpha k R + \sigma ) \right\} ,
\]
which can be checked to be nonnegative at $\theta_+ (R)$
\[
(P^{tot}_{k,R})'' (\theta_+ (R))= 2 \alpha k R^3 f(R) \geq 0 \,.
\]
This proves that $\theta_{k,R}:=\theta_{+} (R)$ is the unique minimizer of
$P^{tot}_{k,R} $ in $[\alpha^{-1},1]$, for $R$ sufficiently large. Moreover from \eqref{expr} we conclude that
 $\theta_{k,R} \to \alpha^{-1}$ as $R \to + \infty$. 

Evaluating $P^{el}_{k,R}$ and $E^{pl}_R$ at $\theta = \theta_{k,R} $ we find
\begin{align} \label{plug1}
& P^{el}_{k,R} (\theta_{k,R})= \frac{2}{27 \alpha^4 k^2} \left\{ 2 \sigma^3 +  2 \alpha k \sigma^2 (3+f)  R + \left( 2 \alpha^2 k^2 \sigma f \right) R^2 +
\alpha^3 k^3 (1-f)  R^3 \right\}, \\
\label{plug1bis}
& E^{pl}_R(\theta_{k,R})= \sigma R^2 (1- \theta_{k,R}^2).
\end{align}
In order to show \eqref{taylorthm} we perform a Taylor expansion in \eqref{exprf} and \eqref{expr}. 
Using $\sqrt{1+x}=1+x/2-x^2/8+x^3/16 +o(x^3)$ we compute \begin{equation} \label{taylor}
f (R)  = 1 + 2 \left( \frac{ \sigma}{ \alpha k } \right) \frac{1}{R} - \frac{3}{2} \, \left( \frac{\sigma^2}{\alpha^2 k^2} \right) \frac{1}{R^2} + 3\left(  \frac{ \sigma^3}{\alpha^3 k^3} \right) \frac{1}{R^3}    +o \left( \frac{1}{R^3} \right) \,.
\end{equation}
Using \eqref{taylor} we can expand the terms in \eqref{plug1} to get
\begin{align}
	  2 \alpha k \sigma^2 (3+f)  R = (8 \alpha k \sigma^2 ) R + 4 \sigma^3 + o \left( {R} \right)   \label{plug2}  \\
	\left( 2 \alpha^2 k^2 \sigma f \right) R^2 = (2 \alpha^2 k^2 \sigma )R^2 + (4 \alpha k \sigma^2 ) R - 3 \sigma^3 + o \left( {R} \right)   \label{plug3} \\
	\alpha^3 k^3 (1-f) R^3 = -(2 \alpha^2 k^2 \sigma )R^2 + \frac{3}{2} ( \alpha k \sigma^2 ) R  - 3 \sigma^3 + o \left( {R} \right)  \label{plug4}
\end{align}
Plugging \eqref{plug2}-\eqref{plug4} into \eqref{plug1} yields the first equation in \eqref{taylorthm}.
Next we compute 
\begin{equation} \label{plug5} 
\theta_{k,R}^2 = \frac{1 }{9 \alpha^2} \left\{ 5 + 4 f + 2 c (4+f) \frac{1}{R}+ \frac{2 c}{R^2} \right\} \, .
\end{equation}
Plugging \eqref{taylor} into \eqref{plug5} gives
\begin{equation} \label{plug6}
\theta_{k,R}^2=\frac{1 }{ \alpha^2} \left\{ 1 +  \frac{2c}{R} + o \left( \frac{1}{R^3} \right) \right\} \,.
\end{equation}
The second relation in  \eqref{taylorthm}
follows by inserting 
 \eqref{plug6} into \eqref{plug1bis}. 

The last part of the statement follows from \eqref{finalst}, taking into account that $\theta_{k,R} \to {\alpha}^{-1}$ and $k_R^{el} \to C^{el}$ as $R \to + \infty$; the details are 
safely left to the reader. 
\end{proof}

\section{Conclusions and perspectives}
In this paper we have proposed a simple continuous model for dislocations at semi-coherent interfaces. 
Our analysis seems flexible enough to describe different  interfaces and several crystalline configurations. Here we discuss the main achievements of this paper, possible extensions to other physical systems, and future perspectives.

\subsection{Main achievements}
In the first part of the paper we have analyzed a line tension model for dislocations at semi-coherent interfaces, in the context of nonlinear elasticity. Within this model, we have shown that there exists a critical size of the crystal such that dislocations become energetically more favorable than purely elastic deformations. More precisely,
we have shown that the energy induced by dislocations scales as the surface area of the interface, while the purely elastic energy scales as the volume of the crystal.  This reflects the fact that dislocations form periodic networks at the interface. Even without giving a rigorous proof of this fact, we have made an explicit construction of a periodic array of dislocations, which is optimal in the scaling of the energy. 

Once proved that the energy scales as the surface of the interface, 
 we have proposed a simpler and more specific continuous model for dislocations, describing dislocations between two 
 crystals with different lattice parameters and, to some extent,  dislocations  in heterostructured nanowires and in epitaxial crystal growth. 
 In such a model the area of the reference configuration of the overlayer is a free parameter, while in the deformed configuration there is a perfect match between the underlayer and the overlayer. 
 
  The variational formulation seems to be quite elegant and effective. As a matter of fact  it is very basic, depending only on three parameters: the diameter of the underlayer, the misfit between the lattice parameters, and the free boundary, described by a single parameter: the area gap between the reference underlayer and overlayer, tuning the amount of dislocations at the interface. 
 If the interface is saturated with  dislocations, then the energy reduces to the surface dislocation energy, while if no dislocation is present, then the energy reduces to the volume elastic energy 
 resulting from the incompatibility of the stress free configurations for the two crystal structures. 
 
The proposed variational model is   
reach enough to describe the size effects already discussed. Indeed, we have show that, in the limit $R\to +\infty$, the surface energy induced by dislocations is predominant (scaling like $R^2$), while the volume elastic energy represents a lower order term (scaling like $R$). Since the elastic energy is vanishing, we can linearize the problem: The asymptotic behaviour of the total energy functional is explicit, depending only on the material parameters in the energy functional, and on the linearized elastic tensor. 

The only unknown  parameter in our formulation is $\sigma$, which roughly speaking (multiplied by $b$) represents the energy per  unit dislocation length  (while $\sigma (\alpha^2-1)$ represents the energy per unit  area of the interface). 
We have proposed some explicit formula for $\sigma$, depending only on the elastic tensor and on a core energy parameter $\gamma^{ch}$, describing the core (chemical) energy per unit dislocation length (see \eqref{sigma?}).

Notice that, since our minimizers are almost explicit, in principle one could fit such parameters through experimental data. In particular, one could implicitly measure $\sigma$  
and $\gamma^{ch}$ as follows. In our model, the lower base of the overlayer is, in the  
deformed configuration, a square of side $R$, while the free upper base (being stress free, at least for large height $h$) has size $R^{up}\approx \alpha \theta R$. 
Such quantity is  observable, and together with the explicit expression \eqref{plug6}  for the optimal $\theta$, 
implicitly determines $\sigma$ (and thus $\gamma^{ch}$) as follows
$$
\sigma= \frac{C^{el}}{2} \left( \frac{(R^{up})^2}{\alpha R} - \frac{R}{\alpha^2}\right).
$$ 

\subsection{Perspectives}
We have worked in a bounded domain $\Omega=S_r \times [0,hr]$, suited to describe nanowires as well as epitaxial crystal growth.  
We point out that all the results are true also if $h=+\infty$, with minor technical changes.

Moreover, using the same ideas from the previous section, we can model different crystal configurations. For instance, consider  two wires $N_{\rm int}$ and $N_{\rm ext}$ one inside the other. Specifically, the external wire can be represented by $(S_{2R}\setminus S_{R}) \times (0,h R)$ and the internal by $S_{\theta R} \times (0,hR)$ with  $\theta \in [\alpha^{-1},1]$. Here $h>0$ is a fixed height and $\alpha I$ is the equilibrium of $N_{\rm int}$, with $\alpha >1$. The external wire is already in equilibrium. The admissible  deformations of $N_{\rm int}$ are  maps $v \colon N_{\rm int} \to \rtre$ such that $v=\theta^{-1} x$ on
the lateral boundary of $N_{\rm int}$, so that it  matches the internal lateral boundary of $N_{\rm ext}$.

As in the previous models, the total energy is given by the sum of an elastic term and a plastic term, the latter proportional to the reference surface mismatch between
the lateral boundaries of the nanowires:
\begin{equation} \label{elastic2}
E^{tot}(v,\theta) = \int W (\nabla v) \, dx +  \sigma h R^2 (1-\theta) \,.
\end{equation}
If $\theta=1$ the two wires coincide and the energy is entirely elastic. If
$\theta=\alpha^{-1}$ then the elastic energy has minimum zero and $E^{tot}$ is purely dislocation energy. If $\theta \in (\alpha^{-1},1)$ then none of the two contributions is zero and we are in a mixed case. 
For such physical system we can carry on the same analysis of Section 2, up to very minor changes.

Our model could be used to describe epitaxial crystal growth: the total energy in this context should be completed by adding the surface energy induced by the exterior boundary of the overlayer. 

Another challenging investigation concerns the energy induced by dislocations at grain boundaries, where the misfit between the crystal lattices are described by rotations rather than dilations. The basic example is given by 
a partition of  a reference polycrystal $\Om$ into grains $\Om_i$, 
where each grain has an orientation determined by a rotation $R_i$ of a reference lattice. 
As in our model, the shape of the grain $\Om_i$ should not be prescribed, so that the variational problem involves free boundaries. Also the rotations $R_i$ could enter the minimization process, which should be completed by suitable boundary conditions or forcing terms. 

\subsection{Conclusions}
This paper proposes a basic  variational model  describing the competition between the plastic energy spent at interfaces, and the corresponding release of bulk energy. 
In  this variational formulation, the size of the interface is a free parameter. In this respect, our model fits into the class of so called free boundary problems. 

The proposed  energy is built upon some heuristic arguments, supported by 
formal mathematical derivations based on the semi-discrete theory of dislocations. 

While the paper focuses on a specific configuration, the method seems flexible to be extended to several  crystalline structures and to different physical contexts, such as grain boundaries and epitaxial growth.  
Depending on the specific context under examination, the total energy functional could be enriched, taking into account external surface energies, mutual rotations between the lattice structures, forcing terms and boundary conditions.

\nocite{*}
\bibliography{bibliografia}

\bibliographystyle{plain}

\end{document}